\documentclass{amsart}

\RequirePackage[OT1]{fontenc}
\RequirePackage{amsthm,amsmath}
\RequirePackage[numbers]{natbib}
\RequirePackage[colorlinks,citecolor=blue,urlcolor=blue]{hyperref}
\usepackage{graphicx}
\usepackage{latexsym,color}
\usepackage{graphicx}
\usepackage{amssymb}
 \usepackage{amsfonts}
 \usepackage{amsmath}
\usepackage{amsthm}
\newtheorem{thm}{Theorem}[section]
 
 \newtheorem{lem}[thm]{Lemma}
 \newtheorem{rem}[thm]{Remark}
\newtheorem{cor}[thm]{Corollary}

\numberwithin{equation}{section}

\begin{document}

\title{Numerical Scheme for Dynkin Games under Model Uncertainty}
\thanks{}

\author{Yan Dolinksky and Benjamin~Gottesman }

\date{\today}

\maketitle \markboth{}{}
\renewcommand{\theequation}{\arabic{section}.\arabic{equation}}
\pagenumbering{arabic}

\begin{abstract}\noindent
We introduce an efficient numerical scheme for continuous time Dynkin
games under model uncertainty. We use the Skorokhod embedding in order to
construct recombining tree approximations. This technique allows us to
determine convergence rates and to construct numerically optimal stopping strategies.
We apply our method to several examples of game options.
\end{abstract}

\section{Introduction}\label{sec:1}\setcounter{equation}{0}
In this paper, we propose an efficient numerical scheme for the computations of
values of Dynkin games under volatility uncertainty.
We consider a finite maturity, continuous--time robust Dynkin game with respect to a non dominated set of mutually singular
probabilities on the canonical space of continuous paths. In this game, Player 1 who negatively/conservatively
thinks that the nature is also against him, will pay the following payment to Player 2 if the two players choose
stopping strategies $\gamma$ and $\tau$ respectively,
\begin{equation}\label{1.1}
H(\gamma,\tau):=\mathbb{I}_{\gamma<\tau}X_{\gamma}+\mathbb{I}_{\tau\leq\gamma}Y_{\tau}+\int_{0}^{\gamma\wedge\tau}Z_u du.
\end{equation}
We model uncertainty by assuming that the stochastic processes $X,Y,Z$ are path--independent
functions of an underlying asset $S$ which is an exponential martingale with volatility
in a given interval.
Thus, our setup can be viewed as
a Dynkin game variant of Peng's G--expectation (see \cite{P}).

For finite maturity optimal stopping problems/games,
there are no explicit solutions even in the relatively simple framework where the
probabilistic setup is given and the payoffs are path--independent functions of the standard Brownian motion.
Hence, numerical schemes come naturally into the picture.

In \cite{ALP}, the authors presented a recombining trinomial tree based
approximations for what is now known as a $G$--expectation in the sense of Peng (\cite{P}).
However, they did not provide a rigorous proof
for the convergence of their scheme and did not obtain error estimates.
Moreover, a priori, it is not clear whether the tree approximations from \cite{ALP}
can be applied for optimal stopping problems/games.

In this paper, we modify slightly the trinomial trees from \cite{ALP}. For the modified (recombining) trees
we construct a discrete time version of the Dynkin game given by (\ref{1.1}).
The main idea is to apply the Skorokhod embedding technique
in order to prove the existence of an
exact scheme along stopping times with the required properties. More precisely, for any
exponential martingale with volatility
in a given interval we prove that there exists
a sequence
of stopping times such that the ratio of the martingale between two sequel times
belongs to some fixed set of the form
$\left\{\exp\left(-\bar\sigma\sqrt\frac{T}{n}\right),1,\exp\left(\bar\sigma\sqrt\frac{T}{n}\right)\right\}$
and the expectation of the
difference between two sequel times is approximately equal to $\frac{T}{n}$.
Here $\bar\sigma>0$ is the right endpoint of the volatility uncertainty interval, $n$ is the number of time steps
and $T$ is the maturity date. This machinery also allows to go in the reverse direction,
namely for a given distribution on the trinomial tree
we can find a "close" distribution on the canonical space which lies in our set of model uncertainty.

We prove the convergence of the
discrete time approximations to the original control problem. Moreover, we provide
error estimates of order $O(n^{-1/4})$.
The recombining structure
of the trinomial trees allows to compute
the corresponding value with complexity $O(n^2)$ where $n$ is the number of time steps.

The idea of
using the Skorokhod embedding technique
in order to obtain an
exact sequence along stopping times
was also employed in a recent work \cite{BDG} where the authors approximated a one dimensional time--homogeneous
diffusion
by recombining trinomial trees (and obtained the same error of order $O(n^{-1/4})$).
In \cite{BDG}, the authors were able to construct explicitly the stopping times. The construction relies heavily on the
well established theory for exit times of one dimensional time--homogeneous
diffusion processes.
This theory cannot be applied in the present work, since the martingales in the volatility uncertainty
setup are not
necessarily diffusions, or even Markov processes.
Thus, the case of model uncertainty requires additional machinery which we develop in
Section \ref{sec:3}. Moreover, since the martingales may not be Markovian we cannot provide an explicit construction of the stopping times
(as done in \cite{BDG}),
but only prove their existence.

Let us remark that the multidimensional version of the above described result is an open question which requires
a completely different approach. In particular, it is not clear how to derive recombining tree models
which will approximate volatility uncertainty in the multidimensional setup.
We leave this challenging question for future research.

Since its introduction in \cite{Dyn}, Dynkin games have been analyzed in discrete and continuous time
models for decades (see, for instance, \cite{BF,BS,LM,N,O}).
In Mathematical Finance, the theory of Dynkin games can be applied to pricing and hedging game options
and their derivatives, see \cite{D,HZ,K,KK,MC}
and the references in the survey paper \cite{Ki}.
In particular, the
nondominated version of the optional decomposition
theorem developed in \cite{Nu} provides a direct link
(as we will
see rigorously)
between Dynkin games and
pricing game options in the model uncertainty
framework.
In general, the theme of Dynkin games
is a central topic in stochastic control.

In \cite{CK}, the authors connected
Dynkin games to
backward stochastic differential equations (BSDEs)
with two reflecting barriers. This link inspired a very active
research in the field of Dynkin games in a Brownian framework,
see e.g. \cite{BL,BY1,H,HH1,HH,X}.
Motivated by Knightian uncertainty, recently there is also
a growing interest in Dynkin games under model uncertainty, see \cite{BY,D,HH,Y}).
In \cite{BY} the authors analyzed a robust version of the Dynkin game over a set of mutually singular probabilities.
They proved that the game admits a value. Moreover, they established submartingale properties of the value process. These
results will be essential in the present work.

The rest of the paper is organized as follows. In the next section we formulate our main result (Theorem \ref{thm2.1}).
In Section 3, we introduce our main tool which is Skorokhod embedding under model uncertainty.
In Section 4, we complete the proof of the main result. Section 5 is devoted to some
auxiliary estimates which are used in the proof of Theorem \ref{thm2.1}.
In Section 6, we provide numerical analysis
for several examples of game options. Moreover, we argue rigorously
the link between Dynkin games and pricing of game options
in the current setup of model uncertainty.

\section{Preliminaries and Main Result}\label{sec:2}\setcounter{equation}{0}
Let $\Omega:=C(\mathbb R_{+}, \mathbb R)$
be the space of
continuous paths equipped with the topology of locally uniform convergence
and the Borel $\sigma$--field $\mathcal F:=\mathcal B(\Omega$). We denote by $B=B_t$, $t\geq 0$
the canonical process $ B_t(\omega):=\omega_t$ and by $\mathcal F=\mathcal F_t$, $t\geq 0$
the natural filtration generated by
$B$.
For any $t$, $\mathcal T_t$ denotes the set of all stopping times with values in $[0,t]$. We denote
by $\mathcal T$ the set of all stopping times (we allow the stopping times to take the value $\infty$).

For a closed interval $I=[\underline{\sigma},\overline{\sigma}]\subset {\mathbb R}_{+}$ and $s>0$ let
$\mathcal P^{(I)}_s$ be the set of all probability measures $P$ on $\Omega$ under which the canonical
process $B$ is a strictly positive martingale
such that $B_0=s$ $P$--a.s.,
the quadratic variation $\langle  B\rangle $ is absolutely continuous $dt\otimes P$ a.s.  and
$B^{-1}_t\sqrt\frac{d\langle  B\rangle _t}{dt}\in I$ $dt\otimes P$ a.s. Observe that if we define the local martingale
$M_t:=\int_{0}^t \frac{d B_u}{B_u}$, then from It\^{o} Isometry we get
$\sqrt\frac{d\langle M\rangle_t}{dt}=B^{-1}_t\sqrt\frac{d\langle  B\rangle _t}{dt}\in I$. Thus $M$ is a true martingale
and $B_t=\exp(M_t-\langle M\rangle _t/2)$, $t\geq 0$ is the
Dol\'{e}ans--Dade exponential
of $M$. In other words, the set
$\mathcal P^{(I)}_s$ is the set of all probability measures (on the canonical space) such that the canonical process
(which starts in $s$) is a Dol\'{e}ans--Dade exponential of a true martingale with volatility in the interval $I$.

From mathematical finance point of view, the set $\mathcal P^{(I)}_s$
describes the set of all possible distributions of the (discounted) stock price process.
We assume that $I$ is a finite interval, i.e.
$\overline{\sigma}<\infty$. This implies that the set $\mathcal P^{(I)}_s$ is weakly compact and so we can apply the results form \cite{BY}
related to the existence of the optimal strategy of the Dynkin game.
Moreover, the assumption $\overline{\sigma}<\infty$ is essential for
constructing an appropriate sequence of trinomial models .
In addition, we assume that
$\underline{\sigma}>0$, in other words
the model uncertainty setup is "noisy enough". This assumption is technical and will be needed for
obtaining uniform bounds on the expectation of the hitting times
related to the canonical process.

We consider a Dynkin game with maturity date $T<\infty$ and
a payoff given by (\ref{1.1}) with
$X_t=g(t,B_t)$, $Y_t=f(t,B_t)$,
$Z_t=h(t,B_t)$
where $g,f,h:[0,T]\times\mathbb R_{+}\rightarrow \mathbb R$ satisfy
$g\geq f$ and the following Lipschitz condition
\begin{eqnarray}\label{2.function}
&|f(t_1,x_1)-f(t_2,x_2)|+|g(t_1,x_1)-g(t_2,x_2)|+|h(t_1,x_1)-h(t_2,x_2)|\leq \\
&L \left((1+|x_1|)|t_2-t_1|+|x_2-x_1|\right), \ t_1,t_2\in [0,T], \ x_1,x_2\in \mathbb R_{+} \nonumber
\end{eqnarray}
for some constant $L$.

For any $(t,x)\in [0,T]\times\mathbb R_{+}$ define the lower value and the upper value of the game at time $t$ given that the canonical
process satisfies $B_t=x$
\begin{eqnarray*}
&\underline{V}^{(I)}(t,x):=
\sup_{P\in \mathcal{P}^{(I)}_x}\sup_{\tau\in\mathcal T_{T-t}}\inf_{\gamma\in\mathcal T_{T-t}}E_{P}[g(\gamma+t, B_{\gamma}){\mathbb I}_{\gamma<\tau}+\\
&f(\tau+t,B_{\tau})\mathbb{I}_{\tau\leq\gamma}
+\int_{0}^{\gamma\wedge\tau}h(u+t, B_{u})du ]
\end{eqnarray*}
and
\begin{eqnarray*}
&\overline{V}^{(I)}(t,x):=\inf_{\gamma\in\mathcal T_{T-t}}\sup_{ P\in \mathcal{P}^{(I)}_x}\sup_{\tau\in\mathcal T_{T-t}}E_{P}
[g(\gamma+t,B_{\gamma})\mathbb{I}_{\gamma<\tau}+\nonumber\\
&+f(\tau+t,B_{\tau})\mathbb{I}_{\tau\leq\gamma}+\int_{0}^{\gamma\wedge\tau}h(u+t,B_{u})du].
\nonumber
\end{eqnarray*}
From Theorem 4.1 in \cite{BY} it follows that the lower value and the upper value coincide and thus the game has a value
\begin{equation}\label{2.2-}
V^{(I)}(t,x):=\overline{V}^{(I)}(t,x)=\underline{V}^{(I)}(t,x), \ \ \forall(t,x)\in [0,T]\times\mathbb R_{+}.
\end{equation}
Our goal is to calculate numerically the value $V^{(I)}(0,s)$. Moreover, from Theorem 4.1 in \cite{BY} it follows that
the stopping time
$\gamma^{*}:=T\wedge\inf\{t: g(t,B_t)=V^{(I)}(t,B_t)\}$ is an optimal exercise time for Player 1. In Section 6,
we use this formula for numerical calculations
of Player 1's optimal strategy.
\begin{rem}\label{rem2.1}
Our setup is slightly different from the one considered in \cite{BY}. If we use our notations,
then the control problem studied in \cite{BY}
is
\begin{equation}\label{BY1}
\inf_{P\in \mathcal{P}^{(I)}_x}\inf_{\gamma\in\mathcal T_{T}}\sup_{\tau\in\mathcal T_{T}}
E_P\left[\mathbb{I}_{\gamma<\tau}X_{\gamma}+\mathbb{I}_{\tau\leq\gamma}Y_{\tau}+\int_{0}^{\gamma\wedge\tau}Z_u du\right].
\end{equation}
Theorem 4.1 in \cite{BY} shows that the above infimum and supremum can be exchanged.
Furthermore, the authors showed that
 $\tau^{*}:=T\wedge\inf\{t: Y_t=V^{(I)}(t,B_t)\}$ is an optimal stopping time for Player 2
 which can be viewed as the holder of the corresponding game option. The term given
 in (\ref{BY1}) is the lowest arbitrage free price of the corresponding game option.

Clearly, if we replace $X,Y,Z$ by $-Y,-X,-Z$ and replace $\gamma\leftrightarrow\tau$, then
 the above control problem is equivalent to
 \begin{equation}\label{BY2}
 \sup_{P\in \mathcal{P}^{(I)}_x}\sup_{\tau\in\mathcal T_{T}}\inf_{\gamma\in\mathcal T_{T}}
E_P\left[\mathbb{I}_{\gamma\leq \tau}X_{\gamma}+\mathbb{I}_{\tau<\gamma}Y_{\tau}+\int_{0}^{\gamma\wedge\tau}Z_u du\right].
\end{equation}
This is almost the same control problem as we consider, up to the following change. In our setup, on the event
$\{\gamma=\tau\}$ Player 1 pays the low payoff $Y_{\tau}+\int_{0}^{\tau} Z_u du$
while in
(\ref{BY2}) Player 1 pays the high payoff $X_{\gamma}+\int_{0}^{\gamma} Z_u du$. Still,
Theorem 4.1 in \cite{BY} can be extended to this setup as well by following the same proof.
Furthermore, analogously,
the optimal exercise time for Player 1 is given by
$\gamma^{*}:=T\wedge\inf\{t: X_t=V^{(I)}(t,B_t)\}$. Namely, Theorem 4.1 in \cite{BY} provides an optimal exercise time for the player which plays against nature.
In our setup, this is Player 1 who can be seen as the seller of the game option.
The term given by (\ref{2.2-}) is the highest arbitrage free price of the game option.
\end{rem}

Next, we describe the trinomial models and the main result.
Fix $n\in\mathbb N$.
Let
$\xi^{(n)}_1,...,\xi^{(n)}_n$ be random variables with values in the set
$\{-1,0,1\}$ and let $\mathcal F^{(n)}=\{\mathcal F^{(n)}_k\}_{k=0}^n$ be the filtration
generated by $\xi^{(n)}_k$, $k=0,1,...,n$. Denote by $\mathcal T_n$ the set of all stopping times (with respect
to the filtration $\mathcal F^{(n)}$)
with values in
the set $\{0,1,...,n\}$.

For a given
$t\in [0,T]$ and $s\geq 0$
consider the geometric random walk
\[S^{t,s,n}_k:=s\exp\left(\overline{\sigma}\sqrt\frac{T-t}{n}\sum_{i=1}^k \xi^{(n)}_i\right) \ \ k=0,1,...,n.\]
Clearly, the process $\{S^{t,s,n}_k\}_{k=0}^n$ lies on the grid
$s\exp\left(\overline{\sigma}\sqrt\frac{T-t}{n} i\right)$, $i=-n,1-n,...,0,1,...,n$.
Denote by $\mathcal{P}^{I,t,n}$ the set of all probability measures on $\mathcal{F}^{(n)}_n$
such that for any $k=1,...,n$
\begin{eqnarray}\label{2.1}
&P(\xi^{(n)}_k=1|\mathcal F^{(n)}_{k-1})\in \frac{1}{1+\exp(\overline{\sigma}\sqrt{\frac{T-t}{n}})}\left[\exp\left(-4\overline{\sigma}\sqrt{\frac{T-t}{n}}\right)\underline{\sigma}^2/\overline{\sigma}^2,1\right]\\
&P(\xi^{(n)}_k=-1|\mathcal F^{(n)}_{k-1})=\exp(\overline{\sigma}\sqrt{\frac{T-t}{n}})P(\xi^{(n)}_k=1|\mathcal F^{(n)}_{k-1})\label{2.1+}\\
&P(\xi^{(n)}_k=0|\mathcal F^{(n)}_{k-1})=1- P(\xi^{(n)}_k=1|\mathcal F^{(n)}_{k-1})-P(\xi^{(n)}_k=-1|\mathcal F^{(n)}_{k-1}).\label{2.1++}
\end{eqnarray}
Let us explain the intuition behind the definition of the set $\mathcal{P}^{I,t,n}$.
First, we observe that for any $P\in \mathcal{P}^{I,t,n}$
and $k\geq 1$,
$P(\xi^{(n)}_k=0|\mathcal F^{(n)}_{k-1})\geq 0$, i.e. $P$ is indeed a probability measure.
Moreover, from
(\ref{2.1+})--(\ref{2.1++}) it follows that
for any $k\geq 1$
\begin{eqnarray*}
&E_P\left(\frac{S^{t,s,n}_{k}}{S^{t,s,n}_{k-1}}\big|\mathcal F^{(n)}_{k-1}\right)=
\exp\left(\overline{\sigma}\sqrt{\frac{T-t}{n}}\right)P(\xi^{(n)}_k=1|\mathcal F^{(n)}_{k-1})+\\
&\exp\left(-\overline{\sigma}\sqrt{\frac{T-t}{n}}\right)P(\xi^{(n)}_k=-1|\mathcal F^{(n)}_{k-1})+P(\xi^{(n)}_k=0|\mathcal F^{(n)}_{k-1})=1.
\end{eqnarray*}
Hence,
$\{S^{t,s,n}_k\}_{k=0}^n$ is a martingale
with respect to any probability measure $P\in\mathcal {P}^{I,t,n}$. Finally, from (\ref{2.1})--(\ref{2.1+}) we have that
for any $P\in\mathcal {P}^{I,t,n}$ and $k\geq 1$ the conditional expectation of the ratio of the square of the return and the time step satisfy

\begin{eqnarray*}
&\frac{n}{T-t}E_P \left(\left(\ln{S^{t,s,n}_{k}}-\ln{S^{t,s,n}_{k-1}}\right)^2\big|\mathcal F^{(n)}_{k-1}\right)=\\
&\overline{\sigma}^2
\left(P(\xi^{(n)}_k=1|\mathcal F^{(n)}_{k-1})+P(\xi^{(n)}_k=-1|\mathcal F^{(n)}_{k-1})\right)
=\\
&\overline{\sigma}^2\left(1+\exp\left(\overline{\sigma}\sqrt{\frac{T-t}{n}}\right)\right)P(\xi^{(n)}_k=1|\mathcal F^{(n)}_{k-1})
\in \overline{\sigma}^2\left[\exp\left(-4\overline{\sigma}\sqrt{\frac{T-t}{n}}\right)\underline{\sigma}^2/\overline{\sigma}^2,1\right]\\
&=\left[\underline{\sigma}^2,\overline{\sigma}^2\right]\bigcup \underline{\sigma}^2\left[\exp\left(-4\overline{\sigma}\sqrt{\frac{T-t}{n}}\right),1\right].
\end{eqnarray*}
In the above union of intervals, the first interval is exactly the square of the model uncertainty interval $I$,
and the second interval vanishing as $n\rightarrow\infty$.
This is the reason that we expect
that the set $\mathcal{P}^{I,t,n}$ will be a good approximation of the set $\mathcal P^{(I)}_s$ restricted to the interval $[0,T-t]$. We emphasis that
although the interval $\left[\exp\left(-4\overline{\sigma}\sqrt{\frac{T-t}{n}}\right),1\right]$
is vanishing,
it will be essential for the Skorokhod embedding procedure.

Next, we define the corresponding Dynkin game under model uncertainty.
Introduce the lower value and the upper value of the game
\begin{eqnarray*}
&\underline{V}^{I,n}(t,s):=\\
&\sup_{ P\in \mathcal{P}^{I,t,n}}\max_{\eta\in\mathcal T_n}
\min_{\zeta\in\mathcal T_n} E_{P}[g(t+\zeta (T-t)/n,S^{t,s,n}_{\zeta})\mathbb{I}_{\zeta<\eta}\nonumber\\
&+f(t+\eta (T-t)/n,S^{t,s,n}_{\eta})\mathbb{I}_{\eta\leq\zeta}+\frac{T-t}{n}\sum_{k=0}^{\zeta\wedge\eta-1} h(t+k(T-t)/n,S^{t,s,n}_{k})]
\end{eqnarray*}
and
\begin{eqnarray*}
&\overline{V}^{I,n}(t,s):=\min_{\zeta\in\mathcal T_n}\sup_{ P\in \mathcal{P}^{I,t,n}}\max_{\eta\in\mathcal T_n}
E_{P}[g(t+\zeta (T-t)/n,S^{t,s,n}_{\zeta})\mathbb{I}_{\zeta<\eta}\nonumber\\
&+f(t+\eta (T-t)/n,S^{t,s,n}_{\eta})\mathbb{I}_{\eta\leq\zeta}+\frac{T-t}{n}\sum_{k=0}^{\zeta\wedge\eta-1} h(t+k(T-t)/n,S^{t,s,n}_{k})].\nonumber
\end{eqnarray*}
We argue that the above two values coincide. In \cite{KK}, the authors proved a similar statement for the setup where the set of probability
measures is the set of
 equivalent martingale measures. However, the only property that was used in their proof is that there exists a
 a reference measure. Namely, that there exists a measure $Q$ such that all the probability measures in the model uncertainty set
 are absolutely continuous with respect to $Q$.
 In our case the probability measures in $\mathcal{P}^{I,t,n}$ are defined on a finite sample space which supports the random variables
$\xi^{(n)}_1,...,\xi^{(n)}_n$. Thus, there exists a reference measure
$Q$ for the set $\mathcal{P}^{I,t,n}$. For instance,
take $Q$ to be the probability measure for which $\xi^{(n)}_1,...,\xi^{(n)}_n$ are i.i.d. and taking the values
$-1,0,1$ with the same probability $1/3$. Following the
proof of
Theorem 2.2 in \cite{KK} we conclude that the lower value and the upper value coincide and
so the game has a value
$${V}^{I,n}(t,s):=\overline{V}^{I,n}(t,s)=\underline{V}^{I,n}(t,s) \ \ \forall t,s.$$
Moreover, by using standard dynamical programming for Dynkin games (see \cite{O})
we can calculate $V^{I,n}(t,s)$
by the following backward recursion.
Define the functions
${J}^{I,t,s,n}_k:\{-k,1-k,...,0,1,...,k\}\rightarrow\mathbb R, \ \ k=0,1,...,n.$
\begin{equation}\label{recursion1}
{J}^{I,t,s,n}_n(z):=f\left(T,s \exp\left(\overline{\sigma}\sqrt{\frac{T-t}{n}}z\right)\right).
\end{equation}
For $k=0,1,...,n-1$
\begin{eqnarray}\label{recursion2}
&{J}^{I,t,s,n}_k(z):=\max\bigg(f\left(t+k (T-t)/n,s \exp\left(\overline{\sigma}\sqrt{\frac{T-t}{n}}z\right)\right), \\
&\min \bigg(g\left(t+k(T-t)/n,s \exp\left(\overline{\sigma}\sqrt{\frac{T-t}{n}}z\right)\right),\frac{T-t}{n}h\left(t+k(T-t)/n,S^{t,s,n}_{k}\right)+\nonumber\\
&\sup_{p\in \left[\exp\left(-4\overline{\sigma}\sqrt{\frac{T-t}{n}}\right)\underline{\sigma}^2/\overline{\sigma}^2,1\right]}\left((1-p)J^{I,t,s,n}_{k+1}(z)+
\frac{p}{1+\exp\left(\overline{\sigma}\sqrt{\frac{T-t}{n}}\right)}J^{I,t,s,n}_{k+1}(z+1)\right.\nonumber \\
&\left. +
\frac{p\exp\left(\overline{\sigma}\sqrt{\frac{T-t}{n}}\right)}{1+\exp\left(\overline{\sigma}\sqrt{\frac{T-t}{n}}\right)}J^{I,t,s,n}_{k+1}(z-1)\right)\bigg)\bigg)\nonumber\\
&=\max\bigg(f\left(t+k (T-t)/n,s \exp\left(\overline{\sigma}\sqrt{\frac{T-t}{n}}z\right)\right), \nonumber\\
&\min \bigg(g\left(t+k(T-t)/n,s \exp\left(\overline{\sigma}\sqrt{\frac{T-t}{n}}z\right)\right),\frac{T-t}{n}h\left(t+k(T-t)/n,S^{t,s,n}_{k}\right)+\nonumber\\
&\max_{p\in \left\{\exp\left(-4\overline{\sigma}\sqrt{\frac{T-t}{n}}\right)\underline{\sigma}^2/\overline{\sigma}^2,1\right\}}
\left((1-p)J^{I,t,s,n}_{k+1}(z)+\frac{p}{1+\exp\left(\overline{\sigma}\sqrt{\frac{T-t}{n}}\right)}J^{I,t,s,n}_{k+1}(z+1)\right. \nonumber\\
&\left. +
\frac{p\exp\left(\overline{\sigma}\sqrt{\frac{T-t}{n}}\right)}{1+\exp\left(\overline{\sigma}\sqrt{\frac{T-t}{n}}\right)}J^{I,t,s,n}_{k+1}(z-1)\right)\bigg)\bigg),\nonumber
\end{eqnarray}
where the last equality follows from the fact that the supremum (maximum) on an interval of a linear function (with respect to $p$)
is achieved at the end points.
We get that
\begin{equation}\label{recursion3}
V^{I,n}(t,s)=J^{I,t,s,n}_0(0).
\end{equation}
Hence, we see that the computation of $V^{I,n}$ is very simple and its complexity is $O(n^2)$.
Next, we formulate our main result.
\begin{thm}\label{thm2.1}
There exists a constant $C>0$ such that for all $(t,s)\in [0,T]\times\mathbb R_{+}$,
\[|V^{I,n}(t,s)-V^{(I)}(t,s)|\leq C(1+s) n^{-1/4}.\]
\end{thm}

From (\ref{recursion1})--(\ref{recursion2}) and the backward induction it follows that
for a fixed $n$ the function $J^{I,\cdot,\cdot\cdot,n}_0:[0,T]\times\mathbb R_{+}\rightarrow\mathbb R$ is
continuous. This together with (\ref{recursion3}) and Theorem \ref{thm2.1} gives immediately the following
Corollary.\\
\begin{cor}\label{cor2.3}
The function $V^{(I)}(t,s):[0,T]\times\mathbb R_{+}\rightarrow\mathbb R$
is continuous.
\end{cor}

\section{Skorokhod Embedding under Model Uncertainty}\label{sec:3}\setcounter{equation}{0}
In this section we fix an arbitrary $n\in\mathbb N$ .
For any $A\in (0, \overline{\sigma}\sqrt{T/n}]$ and stopping time $\theta\in\mathcal T$ (recall that $\mathcal T$ is the set of all stopping times
with respect to the canonical filtration)
consider the stopping times
\begin{eqnarray}\label{3.0}
&\rho^{(\theta)}_A:=\inf\{t\geq \theta: |\ln B_t-\ln B_{\theta}|=A\} \ \ \mbox{and}\\
&\kappa^{(\theta)}_A:=\infty\mathbb I_{\rho^{(\theta)}_A=\infty}+\sum_{i=1}^2 (-1)^i \mathbb I_{\ln  B_{\rho^{(\theta)}_A}=\ln B_{\theta}+(-1)^i A}\times\nonumber\\
&\inf\left\{t\geq \rho^{(\theta)}_A: \ln B_t=\ln B_{\theta} \ \mbox{or} \
\ln B_t=\ln B_{\theta}+(-1)^i \overline{\sigma}\sqrt{T/n}\right\},\nonumber
\end{eqnarray}
where the infimum over an empty set is equal to $\infty$.
Set
\[z:=z(n)=\exp(-2\overline{\sigma}\sqrt{T/n})\overline{\sigma}^{-2}
\frac{\exp(2\overline{\sigma}\sqrt{T/n})+\exp(-2\overline{\sigma}\sqrt{T/n})-2}{2+\exp(\overline{\sigma}\sqrt{T/n})+\exp(-\overline{\sigma}\sqrt{T/n})}.\]
Observe that $z=T/n+O(n^{-3/2})$.
As usual, we use the convention $O(x)$ to denote
a random variable ($z(n)$ is deterministic) that is uniformly (in time and space) bounded after dividing by $x$.

We start with the following lemma.
\begin{lem}\label{lem3.1}
Let $P\in \mathcal P^{(I)}_s$ and let $\theta\in\mathcal T$ satisfy $E_P[\theta]<\infty$.
There exists a stopping time $\mathcal T\ni\hat\theta\geq\theta$ such that $P$ a.s. we have
$\hat\theta<\infty$ and
$\frac{B_{\hat\theta}}{B_{\theta}}\in \left
\{\exp(-\overline{\sigma}\sqrt{T/n}),0,\exp(\overline{\sigma}\sqrt{T/n})
\right\}$. Furthermore,
$E_{P}(\hat\theta-\theta|\mathcal F_{\theta})=z$ and
\begin{eqnarray}\label{3.1}
&P\left(\frac{B_{\hat\theta}}{ B_{\theta}}=\exp(\overline{\sigma}\sqrt{T/n})|\mathcal F_{\theta}\right)
\in \frac{1}{1+\exp(\overline\sigma T/n)}\left[\exp\left(-4\overline\sigma T/n\right)\underline{\sigma}^2/\overline{\sigma}^2,1\right],\\
& P\left(\frac{ B_{\hat\theta}}{ B_{\theta}}=\exp(-\overline{\sigma}\sqrt{T/n})|\mathcal F_{\theta}\right)=\exp(\overline{\sigma}\sqrt {T/n})
P\left(\frac{ B_{\hat\theta}}{ B_{\theta}}=\exp(\overline{\sigma}\sqrt{T/n})|\mathcal F_{\theta}\right)\label{3.2},\\
&P\left({ B_{\hat\theta}}={ B_{\theta}}|\mathcal F_{\theta}\right)=1-
 P\left(\frac{B_{\hat\theta}}{B_{\theta}}=\exp(\overline{\sigma}\sqrt{T/n})|\mathcal F_{\theta}\right)\label{3.3}\\
&-P\left(\frac{B_{\hat\theta}}{B_{\theta}}=-\exp(\overline{\sigma}\sqrt{T/n})|\mathcal F_{\theta}\right).\nonumber
\end{eqnarray}
Notice the resemblance to the formulas (\ref{2.1})--(\ref{2.1++}). In particular, (\ref{3.1}) gives the technical reason for the
definition given by (\ref{2.1}).
\end{lem}
\begin{proof}
Denote
$\rho:=\rho^{(\theta)}_{\overline{\sigma}\sqrt{T/n}}$.
From the fact that $B$ is a $P$--martingale with volatility bonded away from zero, it follows that
$E_P[\rho]<\infty$. Thus, $\frac{B_{\rho}}{B_{\theta}}=\exp(\pm \overline{\sigma}\sqrt{T/n})$, $P$--a.s.,
and from the martingale property we have
\begin{equation*}
P\left( B_{\rho}= B_{\theta}\exp(\pm\overline{\sigma}\sqrt{T/n})|\mathcal F_{\theta}\right)=
\frac{1}{1+\exp(\pm \overline{\sigma}\sqrt{T/n})}.
\end{equation*}
Hence,
\begin{equation}\label{3.1000}
E_{P}\left((B_{\rho}-B_{\theta})^2|\mathcal F_{\theta}\right)=
\frac{\exp(2\overline{\sigma}\sqrt{T/n})+
\exp(-2\overline{\sigma}\sqrt{T/n})-2}{2+\exp(\overline{\sigma}\sqrt{T/n})+
\exp(-\overline{\sigma}\sqrt{T/n})}B^2_{\theta}.
\end{equation}
From
 the It\^{o} isometry and the fact that under $P$, the process $B$ is an
 exponential martingale with volatility less or equal then $\overline{\sigma}$ we
 obtain
 $$E_{P}\left((B_{\rho}-B_{\theta})^2|\mathcal F_{\theta}\right)\leq E_P\left[\int_{\theta}^{\rho} B^2_t\overline{\sigma}^2 dt|\mathcal F_{\theta}\right]\leq
\overline{\sigma}^2\exp( 2\overline{\sigma}\sqrt{T/n})  B^2_{\theta}E_{P}(\rho-\theta|\mathcal F_{\theta}),$$
 where the last inequality follows from the fact that $B_t\leq \exp(\overline{\sigma}\sqrt{T/n}) B_{\theta}$ for $t\in [\theta,\rho]$.
This together with (\ref{3.1000}) yields
\begin{equation}\label{3.5}
E_{P}(\kappa^{(\theta)}_{\overline{\sigma}\sqrt{T/n}}-\theta|\mathcal F_{\theta})=E_{P}(\rho-\theta|\mathcal F_{\theta})\geq z.
\end{equation}
Next, we notice that for $A_2>A_1$ we have $\kappa^{(\theta)}_{A_2}>\kappa^{(\theta)}_{A_1}$, $P$ a.s.
Moreover, if $A_n\uparrow A$ then
$\kappa^{(\theta)}_{A_n}\uparrow \kappa^{(\theta)}_A$ $P$ a.s.
Hence, from the Monotone Convergence Theorem
\begin{equation}\label{3.6}
 A_n\uparrow A \  \Rightarrow  \
E_{P}(\kappa^{(\theta)}_A|\mathcal F_{\theta})=\lim_{n\rightarrow\infty}E_{P}(\kappa^{(\theta)}_{A_n}|\mathcal F_{\theta}).
\end{equation}
Let $\mathbb Q$ be the set of rational numbers. Define the random variable
\[\mathcal Z:=\sup\{q\in\mathbb Q\cap (0, \overline{\sigma}\sqrt{T/n}]:E_{P}(\kappa^{(\theta)}_q|\mathcal F_{\theta})\leq z \}.\]
Clearly, $\mathcal Z$ is $\mathcal F_{\theta}$--measurable.
Moreover, from the monotonicity property of $\kappa^{(\theta)}_A$ and (\ref{3.5})--(\ref{3.6}), we obtain for the stopping time
$\hat\theta:=\kappa^{(\theta)}_{\mathcal Z}$
that
$E_{P}(\hat\theta-\theta|\mathcal F_{\theta})=z.$

Finally, from the fact that
$\frac{B_{\hat\theta}}{B_{\theta}}\in \left
\{\exp(-\overline{\sigma}\sqrt{T/n}),0,\exp(\overline{\sigma}\sqrt{T/n})
\right\}$ and
$E_{P}\left(\frac{B_{\hat\theta}}{B_{\theta}}|\mathcal F_{\theta}\right)=1$
we conclude that (\ref{3.2})--(\ref{3.3}) hold true. Thus,
\begin{eqnarray}\label{3.1001}
&E_{P}\left(B^2_{\hat\theta}/B^2_{\theta}-1|\mathcal F_{\theta}\right)=
\left(\exp(2\overline{\sigma}\sqrt{T/n})+\exp(-\overline{\sigma}\sqrt{T/n})\right)\times\\
&P\left(\frac{B_{\hat\theta}}{ B_{\theta}}=\exp(\overline{\sigma}\sqrt{T/n})|\mathcal F_{\theta}\right)-
\left(1+\exp(\overline{\sigma}\sqrt{T/n})\right)P\left(\frac{B_{\hat\theta}}{B_{\theta}}=\exp(\overline{\sigma}\sqrt{T/n})|\mathcal F_{\theta}\right).\nonumber
\end{eqnarray}
By applying the It\^{o} isometry,
we obtain
$$ E_P\left[\int_{\theta}^{\hat\theta} B^2_t\underline{\sigma}^2 dt|\mathcal F_{\theta}\right]\leq E_{P}\left(B^2_{\hat\theta}-B^2_{\theta}|\mathcal F_{\theta}\right)\leq E_P\left[\int_{\theta}^{\hat\theta} B^2_t\overline{\sigma}^2 dt|\mathcal F_{\theta}\right].$$
This together with the equality
$E_{P}(\hat\theta-\theta|\mathcal F_{\theta})=z$
and the inequality $\exp(-\overline{\sigma}\sqrt{T/n}) B_{\theta}\leq B_t\leq \exp(\overline{\sigma}\sqrt{T/n})B_{\theta}$ gives
 $$E_{P}\left(B^2_{\hat\theta}/B^2_{\theta}-1|\mathcal F_{\theta}\right)\in z[\underline{\sigma}^{2}\exp(-2\overline{\sigma}\sqrt{T/n}),\overline{\sigma}^{2}\exp(2\overline{\sigma}\sqrt{T/n})].$$
 Hence, from (\ref{3.1001}) and the definition of $z$ we
conclude (\ref{3.1}) and completes the proof.
\end{proof}

Next, for a given initial stock price $s>0$, we
construct an embedding of probability measures
$\Psi_n:\mathcal P^{I,0,n}\rightarrow \mathcal P^{(I)}_s$.
Choose
$P\in  \mathcal P^{I,0,n}$. There exists functions
\[\phi_i:\{-1,0,1\}^i\rightarrow \frac{1}{1+\exp(\overline{\sigma}\sqrt{T/n})}\left[\exp\left(-4\overline{\sigma}\sqrt{T/n}\right)\underline{\sigma}^2/\overline{\sigma}^2,1\right], \ \ i=0,1,...,n-1\]
such that (\ref{2.1}) holds true with
\begin{equation*}
 P(\xi^{(n)}_k=1|\mathcal F^{(n)}_{k-1})=
\phi_{k-1}(\xi^{(n)}_1,...,\xi^{(n)}_{k-1}), \ \  k=1,...,n.
\end{equation*}
Recall the canonical space $\Omega=C(\mathbb R_{+},\mathbb R)$.
On this sample space we define a sequence of random variables
$A_0,...,A_{n},\theta_0,...,\theta_{n}$ by the following recursion.
Let $\theta_0:=0$ and
$A_0\in (0,\overline{\sigma}\sqrt{T/n}]$ be the unique solution of the equation
\begin{equation*}
\frac{\exp(x)-1}{(1+\exp(x))(\exp(\overline{\sigma}\sqrt{T/n})-1)}=\phi_0.
\end{equation*}
Recall the definition given by (\ref{3.0}). For $k=1,...,n$ set $\theta_k:=\kappa^{(\theta_{k-1})}_{A_{k-1}}$,
and on the event $\{\theta_k<\infty\}$ define $A_{k}\in (0,\overline{\sigma}\sqrt{T/n}]$ to be the unique solution of the equation
\begin{eqnarray*}
&\frac{\exp(x)-1}{(1+\exp(x))(\exp(\overline{\sigma}\sqrt{T/n})-1)}= \\
&\phi_{k}\left(\overline{\sigma}^{-1}(T/n)^{-1/2}(\ln  B_{\theta_1}- \ln B_{\theta_0}),..., \overline{\sigma}^{-1}(T/n)^{-1/2}(\ln B_{\theta_{k}}-\ln B_{\theta_{k-1}})
\right).
\end{eqnarray*}
On the event $\{\theta_k=\infty\}$ we set $A_k=0$.
Define the random variables $\sigma_0,...,\sigma_{n-1}$ by
\begin{eqnarray}\label{3.50}
&\sigma_{k}:=\mathbb{I}_{\theta_k<\infty}\max\bigg(\underline{\sigma},\overline\sigma \sqrt{1+\exp(\overline{\sigma}\sqrt{T/n})}\times\\
&\left(\phi_{k}\left(\overline{\sigma}^{-1}(T/n)^{-1/2}(\ln B_{\theta_1}- \ln B_{\theta_0}),..., \overline{\sigma}^{-1}(T/n)^{-1/2}(\ln B_{\theta_{k}}-\ln  B_{\theta_{k-1}})\right)\right)^{1/2}
\bigg).\nonumber
\end{eqnarray}
Observe that on the event $\{\theta_k<\infty\}$ we have $\sigma_k\in I$.
Thus, the fact that the volatility interval $I$ is bounded away from zero
implies that there exists a unique probability measure $\hat P:=\Psi_n(\Pi)\in \mathcal P^{(I)}_s$
such that $E_{\hat P}[\theta_n]<\infty$, and that for any $k<n$,
$B^{-1}_t\sqrt\frac{d\langle B\rangle _t}{dt}\equiv \sigma_k$
on the random interval $[\theta_k,\theta_{k+1})$ $\hat P$ a.s.
\begin{lem}\label{lem3.2}
The joint distribution of
$\ln B_{\theta_1}- \ln B_{\theta_0},..., \ln B_{\theta_{n}}-\ln  B_{\theta_{n-1}}$ under $\hat P$ is equal to the
joint distribution of $\overline{\sigma}\sqrt {T/n}\xi^{(n)}_1,...,\overline{\sigma}\sqrt {T/n}\xi^{(n)}_{n}$
under $P$. Moreover, for any $k<n$,
$\hat P(B_{\theta_{k+1}}|\mathcal F_{\theta_k})=\hat P(B_{\theta_{k+1}}|B_{\theta_1},...,B_{\theta_k})$ and
$E_{\hat P}(\theta_{k+1}-\theta_{k}|\mathcal F_{\theta_k})=T/n+O(n^{-3/2}).$
\end{lem}
\begin{proof}
For any $k$ we have $\frac{B_{\theta_{k+1}}}{B_{\theta_k}}\in \left
\{\exp(-\overline{\sigma}\sqrt{T/n}),0,\exp(\overline{\sigma}\sqrt{T/n})
\right\}$ and $E_{\hat P}\left(\frac{B_{\theta_{k+1}}}{B_{\theta_k}}|\mathcal F_{\theta_k}\right)=1$.
Fix $k<n$. We argue
that
\begin{eqnarray}\label{3.yan}
&\hat P\left(\frac{B_{\theta_{k+1}}}{B_{\theta_k}}=\exp(\overline{\sigma}\sqrt{T/n})|\mathcal F_{\theta_k}\right)=\\
&\phi_{k}\left(\overline{\sigma}^{-1}(T/n)^{-1/2}(\ln B_{\theta_1}- \ln B_{\theta_0}),..., \overline{\sigma}^{-1}(T/n)^{-1/2}(\ln B_{\theta_{k}}-\ln  B_{\theta_{k-1}})\right).\nonumber
\end{eqnarray}
Indeed, from (\ref{3.0}), the definition of $A_k$ and the
martingale property of $B$ we get
\begin{eqnarray*}
&\hat P\left(\frac{B_{\theta_{k+1}}}{B_{\theta_k}}=\exp(\overline{\sigma}\sqrt{T/n})|\mathcal F_{\theta_k}\right)=\\
&\hat P\left(B_{\rho^{(A_k)}_{\theta_k}}=\exp(A_k) B_{\theta_k}|\mathcal F_{\theta_k}\right)\times\\
&\hat P\left(B_{\theta_{k+1}}=\exp(\overline{\sigma}\sqrt{T/n})B_{\theta_{k}}|B_{\rho^{(A_k)}_{\theta_k}}=\exp(A_k) B_{\theta_k},\mathcal F_{\theta_k}\right)=\\
&\frac{1}{1+\exp(A_k)}\frac{\exp(A_k)-1}{\exp(\overline\sigma\sqrt{T/n})-1}=\\
&\phi_{k}\left(\overline{\sigma}^{-1}(T/n)^{-1/2}(\ln B_{\theta_1}- \ln B_{\theta_0}),..., \overline{\sigma}^{-1}(T/n)^{-1/2}(\ln B_{\theta_{k}}-\ln  B_{\theta_{k-1}})\right)
\end{eqnarray*}
as required. In particular
$\hat P(B_{\theta_{k+1}}|\mathcal F_{\theta_k})=\hat P(B_{\theta_{k+1}}|B_{\theta_1},...,B_{\theta_k})$. Furthermore, from the definition of
$\phi_k$, $k=0,1,...,n-1$ we conclude that the joint distribution of
$\ln B_{\theta_1}- \ln B_{\theta_0},..., \ln B_{\theta_{n}}-\ln  B_{\theta_{n-1}}$ is equal to the
joint distribution of $\overline{\sigma}\sqrt {T/n}\xi^{(n)}_1,...,\overline{\sigma}\sqrt {T/n}\xi^{(n)}_{n}$.

Finally, we estimate $E_{\hat P}(\theta_{k+1}-\theta_{k}|\mathcal F_{\theta_k})$.
From (\ref{3.50}) and the inequality
\[\phi_k\geq \frac{1}{1+\exp(\overline{\sigma}\sqrt{T/n})}\exp\left(-4\overline{\sigma}\sqrt{T/n}\right)\underline{\sigma}^2/\overline{\sigma}^2\] we get
\begin{eqnarray*}
&\phi_{k}\left(\overline{\sigma}^{-1}(T/n)^{-1/2}(\ln B_{\theta_1}- \ln B_{\theta_0}),..., \overline{\sigma}^{-1}(T/n)^{-1/2}(\ln B_{\theta_{k}}-\ln B_{\theta_{k-1}})
\right)=\\
&\sigma^2_k\left(\frac{1}{\overline\sigma^2(1+\exp(\overline\sigma\sqrt{T/n}))}+O(\sqrt{T/n})\right).
\end{eqnarray*}
This together with (\ref{3.2})--(\ref{3.3}) and (\ref{3.yan}) yields
\begin{eqnarray}\label{3.yann}
& E_{\hat P}\left( B^2_{\theta_{k+1}}/B^2_{\theta_k}-1|\mathcal F_{\theta_k}\right)\\
&=\left(\exp(2\overline{\sigma}\sqrt{T/n})+\exp(-\overline{\sigma}\sqrt{T/n})-1-\exp(\overline{\sigma}\sqrt{T/n})\right)\times\nonumber\\
&\sigma^2_k\left(\frac{1}{\overline\sigma^2(1+\exp(\overline\sigma\sqrt{T/n}))}+O(\sqrt{T/n})\right)=\sigma^2_k\left(\frac{T}{n}+O(n^{-3/2})\right).\nonumber
\end{eqnarray}
From the
It\^{o} isometry and the fact that (under the probability measure $\hat P$)
the volatility of the canonical process $B$ is constant (equal to $\sigma_k$) on the interval
$[\theta_k,\theta_{k+1})$ we obtain
$$E_{\hat P}\left( B^2_{\theta_{k+1}}/B^2_{\theta_k}-1|\mathcal F_{\theta_k}\right)\in
\sigma^2_k
E_{\hat P}(\theta_{k+1}-\theta_{k}|\mathcal F_{\theta_k})[\exp(-2\overline\sigma\sqrt{T/n}),\exp(2\overline\sigma\sqrt{T/n})].$$
Thus, from (\ref{3.yann}) it follows that $E_{\hat P}(\theta_{k+1}-\theta_{k}|\mathcal F_{\theta_k})=(1+O(1/\sqrt n))\frac{T}{n}$,
and the proof is completed.
\end{proof}

\section{Proof Theorem \ref{thm2.1}}\label{sec:4}\setcounter{equation}{0}
For simplicity, we assume that the starting time is $t=0$. For a general $t\in [0,T]$ the proof is done in the same way.
Denote by $s>0$ the initial stock price.
\subsection{Proof of the inequality $V^{(I)}(0,s)\leq V^{I,n}(0,s)+C(1+s) n^{-1/4}$}
\begin{proof}
Fix $n\in\mathbb N$ and choose $\epsilon>0$. There exists a probability
measure $P^*\in \mathcal P^{(I)}_s$ and a stopping time $\tau^*\in\mathcal T_T$ such that
\begin{equation}\label{4.1}
V^{(I)}(0,s)\leq \epsilon+ \inf_{\gamma\in\mathcal T_{T}}E_{P^*}
\left[g(\gamma, B_{\gamma}){\mathbb I}_{\gamma<\tau^*}+
f(\tau^*,B_{\tau^*})\mathbb{I}_{\tau^*\leq\gamma}+\int_{0}^{\gamma\wedge\tau^*}h(u, B_{u})du\right].
\end{equation}
From Lemma \ref{lem3.1} it follows that we can choose a sequence of stopping times
$0=\theta_0<\theta_1<\theta_2<...<\theta_n$ such that $P^*$ a.s., for any $i=1,...,n$
\[
\frac{B_{\theta_i}}{B_{\theta_{i-1}}}\in \left
\{\exp(-\overline{\sigma}\sqrt{T/n}),0,\exp(\overline{\sigma}\sqrt{T/n})
\right\},\]
\begin{eqnarray*}
&P^*\left(\frac{B_{\theta_i}}{B_{\theta_{i-1}}}=\exp(\overline{\sigma}\sqrt {T/n})|\mathcal F_{\theta_{i-1}}\right)
\in \frac{1}{1+\exp(\overline{\sigma}\sqrt{T/n})}\left[\exp\left(-4\overline{\sigma}\sqrt{T/n}\right)\underline{\sigma}^2/\overline{\sigma}^2,1\right],\\
& P^*\left(\frac{B_{\theta_i}}{B_{\theta_{i-1}}}=\exp(-\overline{\sigma}\sqrt{T/n})|\mathcal F_{\theta_{i-1}}\right)=\exp(\overline{\sigma}\sqrt{T/n})
P^*\left(\frac{ B_{\theta_i}}{ B_{\theta_{i-1}}}=\exp(\overline{\sigma}\sqrt {T/n})|\mathcal F_{\theta_{i-1}}\right)\label{4.2},\\
& P^*\left({B_{\theta_i}}={ B_{\theta_{i-1}}}|\mathcal F_{\theta_{i-1}}\right)=1-
P^*\left(\frac{B_{\theta_i}}{B_{\theta_{i-1}}}=\exp(\overline{\sigma}\sqrt {T/n})|\mathcal F_{\theta_{i-1}}\right)\label{4.3}\\
&- P^*\left(\frac{B_{\theta_i}}{B_{\theta_{i-1}}}=-\exp(\overline{\sigma}\sqrt {T/n})|\mathcal F_{\theta_{i-1}}\right),\nonumber
\end{eqnarray*}
and
$E_{P^*}(\theta_i-\theta_{i-1}|\mathcal F_{\theta_{i-1}})=z$
where $z=z(n)$ is given before Lemma \ref{lem3.1}.
In words, we apply the Skorokhod embedding technique given by Lemma \ref{lem3.1} in order to
construct a sequence of stopping times
such that the ratio of $B$ between two sequel times
belongs to
$\left\{\exp\left(-\bar\sigma\sqrt\frac{T}{n}\right),1,\exp\left(\bar\sigma\sqrt\frac{T}{n}\right)\right\}$.
Moreover, the expectation of the
difference between two sequel times is $\frac{T}{n}+O(n^{-3/2})$. The last fact
will be used via the Auxiliary Lemmas \ref{lem5.3}--\ref{stoppingtimes}.

Now, comes the main idea of the proof. Recall the geometric random walk
$\{S^{0,s,n}_k\}_{k=0}^n$ and the
trinomial models given by the set of probability measures
$\mathcal P^{I,0,n}$. From (\ref{2.1})--(\ref{2.1++}) and the above properties of the probability measure $P^*$
 it follows that there exists a probability measure $\tilde P\in \mathcal P^{I,0,n}$ such that
 the distribution
 of $\{B_{\theta_i}\}_{i=0}^n$ under $P^{*}$ equals to the distribution
 of $\{S^{0,s,n}_k\}_{k=0}^n$ under $\tilde P$.
Moreover, using similar arguments as in Lemma \ref{lem3.2} we obtain
that for any $k<n$,
$P^{*}(B_{\theta_{k+1}}|\mathcal F_{\theta_k})=P^{*}(B_{\theta_{k+1}}|B_{\theta_1},...,B_{\theta_k}).$
The above two properties give
\begin{eqnarray*}
&\max_{\eta\in\mathcal T_n}
\min_{\zeta\in\mathcal T_n} E_{\tilde P}[g(\zeta T/n,S^{0,s,n}_{\zeta})\mathbb{I}_{\zeta<\eta}\nonumber\\
&+f(\eta T/n,S^{0,s,n}_{\eta})\mathbb{I}_{\eta\leq\zeta}+\frac{T}{n}\sum_{k=0}^{\zeta\wedge\eta-1} h(kT/n,S^{0,s,n}_{k})]=\\
&\sup_{\eta\in\mathcal S_n}\inf_{\zeta\in\mathcal S_n}E_{P^{*}}[g(\zeta T/n,B_{\theta_{\zeta}})\mathbb{I}_{\zeta<\eta}\\
&+f(\eta T/n,B_{\theta_{\eta}})\mathbb{I}_{\eta\leq\zeta}+\frac{T}{n}\sum_{k=0}^{\zeta\wedge\eta-1} h(k T/n,B_{\theta_k})].
\end{eqnarray*}
Hence, we conclude
\begin{eqnarray}\label{4.4}
&V^{I,n}(0,s)\geq\sup_{\eta\in\mathcal S_n}
\inf_{\zeta\in\mathcal S_n}E_{P^{*}}[g(\zeta T/n,B_{\theta_{\zeta}})\mathbb{I}_{\zeta<\eta}\\
&+f(\eta T/n,B_{\theta_{\eta}})\mathbb{I}_{\eta\leq\zeta}+\frac{T}{n}\sum_{k=0}^{\zeta\wedge\eta-1} h(k T/n,B_{\theta_k})].\nonumber
\end{eqnarray}
The final step is technical. We are using (\ref{4.1})--(\ref{4.4}) in order to
bound from above the difference $V^{(I)}(0,s)-V^{I,n}(0,s)$.

Introduce the stopping time $\eta^*:=n\wedge\min\{k:\theta_k\geq\tau^*\}\in\mathcal S_n$.
In view of (\ref{4.4}) there exists a stopping time
$\zeta^*\in\mathcal S_n$ such that
\begin{eqnarray}\label{4.5}
&V^{I,n}(0,s)\geq\\
&E_{P^*}\left[g(\zeta^* T/n,B_{\theta_{\zeta^*}})\mathbb{I}_{\zeta^*<\eta^*}
+f(\eta^* T/n,B_{\theta_{\eta^*}})\mathbb{I}_{\eta^*\leq\zeta^*}+\frac{T}{n}\sum_{k=0}^{\zeta^*\wedge\eta^*-1} h(k T/n,B_{\theta_k})\right]-\epsilon.\nonumber
\end{eqnarray}
Define the stopping time
$\gamma^*:=(T\wedge\theta^{(n)}_{\zeta^*})\mathbb{I}_{\zeta^*<n}+T\mathbb{I}_{\zeta^*=n}\in\mathcal T_T$.
From (\ref{4.1}) and (\ref{4.5}) we obtain that
\begin{eqnarray}\label{4.6}
&V^{(I)}(0,s)\leq V^{(I,n)}(0,s)+2\epsilon+\\
& E_{P^*}[g(\gamma^*,B_{\gamma^*})\mathbb{I}_{\gamma^*<\tau^*}-g(\zeta^* T/n,B_{\theta_{\zeta^*}})\mathbb{I}_{\zeta^*<\eta^*}]\nonumber\\
&+E_{P^*}[f(\tau^*,B_{\tau^*})\mathbb{I}_{\tau^*\leq\gamma^*}
-f(\eta^* T/n,B_{\theta_{\eta^*}})\mathbb{I}_{\eta^*\leq\zeta^*}]\nonumber\\
&+\nonumber E_{P^*}[\int_{0}^{\gamma^*\wedge\tau^*}h(u, B_{u})du-\frac{T}{n}\sum_{k=0}^{\zeta^*\wedge\eta^*-1} h(k T/n,B_{\theta_k})].\nonumber
\end{eqnarray}
From technical reasons we extend the function $h$ to the domain $\mathbb R^2$ by
$h(t,x):=h(t\wedge T,x)$. Clearly, the extended $h$ is satisfying the Lipschitz condition given
by (\ref{2.function}) on the domain $\mathbb R^2$.
We observe that if $\gamma^*<\tau^*$, then $\zeta^*<\eta^*$. This together with (\ref{2.function}),
which in particular implies that $h(t,x)=O(1)(1+|x|)(1+t)$,
and (\ref{4.6})
gives
\begin{eqnarray}\label{4.7}
&V^{(I)}(0,s)\leq V^{(I,n)}(0,s)+2\epsilon+O(1)E_{P^*}|B_{\gamma^*\wedge\tau^*}-B_{\theta_{\zeta^*\wedge\eta^*}}|+\\
&O(1)E_{P^*}\left[(1+\sup_{0\leq t\leq\theta_n\vee T} B_t)(1+\theta_n\vee T)\times\right.\nonumber\\
&\left.(|\gamma^*\wedge\tau^*-\zeta^*\wedge\eta^*\frac{T}{n}|+|\gamma^*\wedge\tau^*-\theta_{\zeta^*\wedge\eta^*}|)\right]\nonumber\\
&+E_{P^*}\left(\max_{1\leq k\leq n}\left|\int_{0}^{\theta_k} h(t,B_t)dt-\sum_{i=0}^{k-1} h(i T/n, B_{\theta_i})\right|\right).\nonumber
\end{eqnarray}
From the definition of the stopping times $\eta^*$ and $\gamma^*$ it follows that
$|\gamma^*\wedge\tau^*-\zeta^*\wedge\eta^* \frac{T}{n}|\leq \max_{1\leq k\leq n}|\theta_k-k T/n|+T/n$
and
\[|\gamma^*\wedge\tau^*-\theta_{\zeta^*\wedge\eta^*}|\leq |T-\theta_n|+\max_{1\leq k\leq n}\theta_k-\theta_{k-1}\leq 3\max_{1\leq k\leq n}|\theta_k-k T/n|+T/n.\]
Hence, from the Cauchy--Schwarz inequality, the Jensen inequality,
Lemma \ref{moments} and Lemma \ref{lem5.3} it follows that
\begin{eqnarray}\label{4.8}
&E_{P^*}\left[(1+\sup_{0\leq t\leq\theta_n\vee T} B_t)(1+\theta_n\vee T)\times\right.\\
&\left.(|\gamma^*\wedge\tau^*-\zeta^*\wedge\eta^*\frac{T}{n}|+|\gamma^*\wedge\tau^*-\theta_{\zeta^*\wedge\eta^*}|)\right]\leq\nonumber\\
&\left(E_{P^*}((1+\sup_{0\leq t\leq\theta_n\vee T} B_t)^4)\right)^{1/4}
\left(E_{P^*}((1+\theta_n\vee T)^4)\right)^{1/4}\times\nonumber\\
&\left(E_{P^*}((4 \max_{1\leq k\leq n}|\theta_k-k T/n|+ 2 T/n)^2)\right)^{1/2}=O((1+s) n^{-1/2}).\nonumber
\end{eqnarray}
Similarly, from the It\^{o} isometry we obtain
\[E_{P^*}((B_{\gamma^*\wedge\tau^*}-B_{\theta_{\zeta^*\wedge\eta^*}})^2)\leq E_{P^*}[\overline\sigma^2\max_{0\leq t\leq\theta_n\vee T}B^2_t
|\gamma^*\wedge\tau^*-\theta_{\zeta^*\wedge\eta^*}|]=O(s
^2 n^{-1/2}).\]
This together with the Jensen inequality, (\ref{4.7})--(\ref{4.8}) and Lemma \ref{stoppingtimes}
gives that
$$V^{(I)}(0,s)\leq V^{I,n}(0,s)+2\epsilon+O((1+s) n^{-1/4})$$
and by letting $\epsilon\downarrow 0$ we complete the proof.
\end{proof}

\subsection{Proof of the inequality $V^{I,n}(0,s)\leq V^{(I)}(0,s)+C(1+s) n^{-1/4}$}
\begin{proof}
The proof is very similar to the proof of the first inequality. Fix $n\in\mathbb N$ and choose $\epsilon>0$.
We abuse notations and denote by $P^{*}$ a probability measure in $\mathcal{P}^{I,0,n}$ which satisfy
\begin{eqnarray}\label{4.9}
&V^{I,n}(0,s)\leq \epsilon+ \max_{\eta\in\mathcal T_n}
\min_{\zeta\in\mathcal T_n} E_{P^*}[g(\zeta T/n,S^{t,s,n}_{\zeta})\mathbb{I}_{\zeta<\eta}\\
&+f(\eta T/n,S^{t,s,n}_{\eta})\mathbb{I}_{\eta\leq\zeta}+\frac{T}{n}\sum_{k=0}^{\zeta\wedge\eta-1} h(kT/n,S^{t,s,n}_{k})].\nonumber
\end{eqnarray}
Recall the definition of $\hat P^{*}:=\Psi_n(P^{*})$
and the stopping times $0=\theta_0<\theta_1<...<\theta_n$ given before Lemma \ref{lem3.2}.
Denote by $\mathcal S_n$ the set of all stopping times with respect to the filtration
$\{\mathcal F_{\theta_i}\}_{i=0}^n$ with values in the set $\{0,1,...,n\}$.
By applying Lemma \ref{lem3.2} and the same arguments as before (\ref{4.4}) it follows that
\begin{eqnarray}\label{4.10}
&\max_{\eta\in\mathcal T_n}
\min_{\zeta\in\mathcal T_n} E_{P^{*}}[g(\zeta T/n,S^{t,s,n}_{\zeta})\mathbb{I}_{\zeta<\eta}\\
&+f(\eta T/n,S^{t,s,n}_{\eta})\mathbb{I}_{\eta\leq\zeta}+\frac{T}{n}\sum_{k=0}^{\zeta\wedge\eta-1} h(kT/n,S^{t,s,n}_{k})]=\nonumber\\
&\sup_{\eta\in\mathcal S_n}
\inf_{\zeta\in\mathcal S_n} E_{\hat P^{*}}[g(\zeta T/n,B_{\theta_{\zeta}})\mathbb{I}_{\zeta<\eta}\nonumber\\
&+f(\eta T/n,B_{\theta_{\eta}})\mathbb{I}_{\eta\leq\zeta}+\frac{T}{n}\sum_{k=0}^{\zeta\wedge\eta-1} h(kT/n,B_{\theta_k})].\nonumber
\end{eqnarray}
The equality (\ref{4.10}) is the cornerstone of the proof. The remaining part is technical, and we use
(\ref{4.9})--(\ref{4.10}) to estimate from above the difference
$V^{I,n}(0,s)-V^{(I)}(0,s)$. Indeed,
from (\ref{4.9})--(\ref{4.10}) it follows that there exists $\eta^{*}\in\mathcal S_n$ (again we abuse notations)
such that
\begin{eqnarray}\label{4.11}
&V^{I,n}(0,s)\leq 2\epsilon+ \inf_{\zeta\in\mathcal S_n}
E_{\hat P^*}[g(\zeta T/n,B_{\theta_{\zeta}})\mathbb{I}_{\zeta<\eta^*}\nonumber\\
&+f(\eta^* T/n,B_{\theta_{\eta^*}})\mathbb{I}_{\eta^*\leq\zeta}+\frac{T}{n}\sum_{k=0}^{\zeta\wedge\eta^*-1} h(kT/n,B_{\theta_k})].\nonumber
\end{eqnarray}
Define the stopping time $\tau^*:=\theta_{\eta^*}\wedge T\in\mathcal T_T$. Clearly, there exists a stopping time $\gamma^{*}\in\mathcal T_T$ such that
\begin{equation*}\label{4.12}
V^{(I)}(0,s)\geq E_{\hat P^*}\left[g(\gamma^*,B_{\gamma^*}){\mathbb I}_{\gamma^*<\tau^*}+
f(\tau^*,B_{\tau^*})\mathbb{I}_{\tau^*\leq\gamma^*}+\int_{0}^{\gamma^*\wedge\tau^*}h(u, B_{u})du\right]-\epsilon.
\end{equation*}
Next, introduce the stopping time
$\zeta^*:=n\wedge\min\{k:\theta_k\geq\gamma^*\}\mathbb{I}_{\gamma^*<T}+n\mathbb{I}_{\gamma^*=T}\in\mathcal S_n$. We observe that
if $\zeta^*<\eta^*$ then $\gamma^*<\tau^*$. Thus, similarly to (\ref{4.7}) we get
\begin{eqnarray*}
&V^{I,n}(0,s)\leq V^{(I)}(0,s)+3\epsilon+O(1)E_{\hat P^*}|B_{\gamma^*\wedge\tau^*}-B_{\theta_{\zeta^*\wedge\eta^*}}|+\\
& O(1) E_{\hat P^*}\left[(1+\sup_{0\leq t\leq\theta_n\vee T} B_t)(1+\theta_n\vee T)(|\gamma^*\wedge\tau^*-\zeta^*\wedge\eta^*\frac{T}{n}|+|\gamma^*\wedge\tau^*-\theta_{\zeta^*\wedge\eta^*}|)\right]\nonumber\\
&+E_{\hat P^*}\left(\max_{1\leq k\leq n}\left|\int_{0}^{\theta_k} h(t,B_t)dt-\sum_{i=0}^{k-1} h(i T/n, B_{\theta_i})\right|\right).\nonumber
\end{eqnarray*}
Finally, by using the same estimates as in Section 4.1, we obtain
that
$$V^{I,n}(0,s)\leq V^{(I)}(0,s)+3\epsilon+O((1+s) n^{-1/4})$$
and by letting $\epsilon\downarrow 0$ we complete the proof.
\end{proof}
\begin{rem}
Let us notice that in the present setup of model uncertainty we get the same error estimates as in the case
with no uncertainty which was studied in \cite{BDG}. The main reason is that Lemma 5.3 which is essential for the proof
cannot be improved even for the most simple case where the canonical process
is a geometric Brownian motion with constant volatility. Namely, the Skorokhod embedding technique cannot provide error estimates of order better than
$O(n^{-1/4})$ even for the approximations of American or game options in the Black--Scholes model.
For details, see \cite{Ki1}.
Fortunately, same estimates can be obtained for the volatility
uncertainty setup.
\end{rem}

\section{Auxiliary Lemmas}\label{sec:5}\setcounter{equation}{0}
In this section we derive the estimates that we used in Section \ref{sec:4}.
We fix $n\in\mathbb N$ and a probability measure $P\in \mathcal P^{(I)}_s$. Furthermore,
we fix a sequence of stopping times
$0=\theta_0<\theta_1<...<\theta_n$ for which we assume that for any $i<n$,
$
\frac{B_{\theta_i}}{B_{\theta_{i-1}}}\in \left
\{\exp(-\overline{\sigma}\sqrt{T/n}),0,\exp(\overline{\sigma}\sqrt{T/n})
\right\}$ $P$--a.s.,
\begin{eqnarray*}
&P\left(\frac{B_{\theta_i}}{B_{\theta_{i-1}}}=\exp(\overline{\sigma}\sqrt {T/n})|\mathcal F_{\theta_{i-1}}\right)
\in \frac{1}{1+\exp(\overline{\sigma}\sqrt{T/n})}\left[\exp\left(-4\overline{\sigma}\sqrt{T/n}\right)\underline{\sigma}^2/\overline{\sigma}^2,1\right],\\
& P\left(\frac{B_{\theta_i}}{B_{\theta_{i-1}}}=\exp(-\overline{\sigma}\sqrt{T/n})|\mathcal F_{\theta_{i-1}}\right)=\exp(\overline{\sigma}\sqrt{T/n})
P\left(\frac{ B_{\theta_i}}{ B_{\theta_{i-1}}}=\exp(\overline{\sigma}\sqrt {T/n})|\mathcal F_{\theta_{i-1}}\right)\label{4.2},\\
& P\left({B_{\theta_i}}={ B_{\theta_{i-1}}}|\mathcal F_{\theta_{i-1}}\right)=1-
P\left(\frac{B_{\theta_i}}{B_{\theta_{i-1}}}=\exp(\overline{\sigma}\sqrt {T/n})|\mathcal F_{\theta_{i-1}}\right)\label{4.3}\\
&- P\left(\frac{B_{\theta_i}}{B_{\theta_{i-1}}}=-\exp(\overline{\sigma}\sqrt {T/n})|\mathcal F_{\theta_{i-1}}\right),\nonumber
\end{eqnarray*}
and
$E_P(\theta_{i+1}-\theta_i|\mathcal F_{\theta_i})=T/n+O(n^{-3/2})$.
Observe that the stopping times $0=\theta_0<\theta_1<...<\theta_n$
from both Section 4.1 and Section 4.2 satisfy the above conditions.

We start with proving the following bound.
\begin{lem}\label{moments}
\[E_P\left(\sup_{0\leq t\leq T\vee\theta_n}B^4_t\right)=O(1)s^4.\]
\end{lem}
\begin{proof}
Clearly, for any $i<n$,
\begin{eqnarray*}
&E_P(B^4_{\theta_{i+1}}-B^4_{\theta_{i}}|\mathcal F_{\theta_i})= B^4_{\theta_{i}}
P\left(\frac{B_{\theta_{i+1}}}{B_{\theta_{i}}}=\exp(\overline{\sigma}\sqrt {T/n})|\mathcal F_{\theta_{i}}\right)\times\\
&\left(\exp(4\overline{\sigma}\sqrt{T/n})-1+\exp(\overline{\sigma}\sqrt{T/n})
(\exp(-4\overline{\sigma}\sqrt{T/n})-1)\right)\leq B^4_{\theta_{i}} O(1/n).
\end{eqnarray*}
Hence, $E_P(B^4_{\theta_{n}})\leq s^4(1+O(1/n))^n= O(1)s^4$. This together with
the Doob inequality gives that
\begin{equation}\label{5.1}
E_P\left(\sup_{0\leq t\leq \theta_n}B^4_t\right)=O(1)s^4.
\end{equation}
Next, we notice that the inequality $B^{-1}_t\sqrt\frac{d\langle B\rangle _t}{dt}\leq\overline\sigma$
together with the
It\^{o} formula implies
that
$\exp(-6\overline\sigma^2 t)B^4_t$, $t\geq 0$ is a super--martingale. In particular,
$E_P B^4_T\leq \exp(6\overline\sigma^2 T)s^4$. Thus, from the
Doob inequality and (\ref{5.1}) we obtain
\[E_P\left(\sup_{0\leq t\leq T\vee\theta_n}B^4_t\right)\leq
E_P\left(\sup_{0\leq t\leq T}B^4_t\right)+E_P\left(\sup_{0\leq t\leq \theta_n}B^4_t\right)=
O(1)s^4\]
and the proof is completed.
\end{proof}
Next, we prove the following.
\begin{lem}\label{lem5.2}
For any $i=0,1,...,n-1$, $E_P((\theta_{i+1}-\theta_i)^4|\mathcal F_{\theta_i})=O(n^{-4})$.
\end{lem}
\begin{proof}
Choose $i<n$. From the Burkholder--Davis--Gundy inequality,
the inequality  $\frac{d\langle B\rangle_t}{dt}\geq\underline\sigma^2 B^2_t$
and the fact that
$\frac{B_t}{B_{\theta_i}}\in [\exp(-\overline\sigma\sqrt{T/n}),\exp(\overline\sigma\sqrt{T/n})]$ for
$t\in [\theta_i,\theta_{i+1}]$
it follows that
\begin{eqnarray*}
&\underline\sigma^8\exp(-8\overline\sigma\sqrt{T/n})B^8_{\theta_i}E_P((\theta_{i+1}-\theta_i)^4|\mathcal F_{\theta_i})\leq\\
& E_P \left((\langle B\rangle_{\theta_{i+1}}-\langle B\rangle_{\theta_{i}})^4|\mathcal F_{\theta_i}\right)=O(1)
 E_P((B_{\theta_{i+1}}- B_{\theta_i})^8|\mathcal F_{\theta_i})=O(n^{-4}) B^8_{\theta_i}
\end{eqnarray*}
and the result follows.
\end{proof}
We arrive to our next estimate.
\begin{lem}\label{lem5.3}
$E_P\left(\max_{0\leq k\leq n}|\theta_k- kT/n|^4\right)= O(n^{-2}).$
\end{lem}
\begin{proof}
Set $Z_i:=\theta_{i}-\theta_{i-1}-E_P(\theta_{i}-\theta_{i-1}|\mathcal F_{\theta_{i-1}})$, $i=1,...,n$.
We use the fact that the expectation of the difference between two sequel times equals approximately to the time step.
Formally, for any $i$, we have $E_P(\theta_{i}-\theta_{i-1}-T/n|\mathcal F_{\theta_{i-1}})=O(n^{-3/2})$. Hence,
\[\max_{0\leq k\leq n}|\theta_k- kT/n|= O(n^{-1/2})+\max_{1\leq k\leq n}|\sum_{i=1}^{k}Z_i|.\]
In view of the inequality $(a+b)^4\leq 8(a^4+b^4)$, $a,b\geq 0$ it remains to prove that
$E_P \left(\left(\max_{1\leq k\leq n}|\sum_{i=1}^{k}Z_i|\right)^4\right)=O(n^{-2})$.
From the Jensen inequality and Lemma \ref{lem5.2} it follows that
$E_P\left((E_P(\theta_{i}-\theta_{i-1}|\mathcal F_{\theta_{i-1}}))^4\right)=O(n^{-4})$ for all $i$.
This together with the inequality $(a-b)^4\leq a^4+b^4,$ $a,b\geq 0$ implies that
$E_P[Z^4_i]=O(n^{-4})$ for all $i$.
Thus, from the Burkholder--Davis--Gundy inequality
applied to the martingale $\sum_{i=1}^k Z_i$, $k=1,...,n$, and the inequality
$\left(\sum_{i=1}^n a_i\right)^2\leq n\left(\sum_{i=1}^n a^2_i\right)$, $a_1,...,a_n\geq 0$,
we obtain
\[E_P \left(\left(\max_{1\leq k\leq n}|\sum_{i=1}^{k}Z_i|\right)^4\right)=O(1)E_P\left(\left(\sum_{i=1}^n Z^2_i\right)^2\right)=O(n)\sum_{i=1}^n E_P Z^4_i=O(n^{-2})\]
as required.
\end{proof}
We end this section with proving the next estimate.
\begin{lem}\label{stoppingtimes}
\[E_P\left(\max_{0\leq k\leq n}\left|\int_{0}^{\theta_k} h(t,B_t) dt-\frac{T}{n}\sum_{i=0}^{k-1}  h(i T/n, B_{\theta_i})\right|\right)= O((1+s) n^{-1/2}).\]
\end{lem}
\begin{proof}
Clearly,
\[\max_{0\leq k\leq n}\left|\int_{0}^{\theta_k} h(t,B_t) dt-\frac{T}{n}\sum_{i=0}^{k-1} h(i T/n, B_{\theta_i})\right|\leq J_1+J_2+\theta_n J_3\]
where
\begin{eqnarray*}
&J_1:=\max_{1\leq k\leq n}\left|\sum_{i=0}^{k-1}h(iT/n, B_{\theta_i})
\left(E_P(\theta_{i+1}-\theta_i|\mathcal F_{\theta_i})-T/n\right)\right|,\\
&J_2:=
\max_{1\leq k\leq n}\left|\sum_{i=0}^{k-1} h(iT/n, B_{\theta_i})
\left(\theta_{i+1}-\theta_i-E_P(\theta_{i+1}-\theta_i|\mathcal F_{\theta_i})\right)\right|,\\
&\mbox{and} \ J_3:=\left(\max_{0\leq k\leq n-1}\sup_{\theta_k\leq t\leq\theta_{k+1}}|h(t,B_t)-h(k T/n, B_{\theta_k}|\right).
\end{eqnarray*}
We have $E_P(\theta_{i+1}-\theta_i|\mathcal F_{\theta_i})=T/n+O(n^{-3/2})$. Hence, from the bound
$h(t,x)=O(1)(1+|x|)(1+t)$, Lemma \ref{moments} and the Jensen inequality it follows that
\[E_P[J_1]= O(n^{-1/2})E_P(1+\max_{0\leq k\leq n-1}B_{\theta_k})
=O((1+s)n^{-1/2}).\]
Next, we estimate $J_2$. We observe that the stochastic process
\[\sum_{i=0}^{k-1} h(iT/n, B_{\theta_i})
\left(\theta_{i+1}-\theta_i-E_P(\theta_{i+1}-\theta_i|\mathcal F_{\theta_i})\right), \ \ k=1,...,n\]
is a martingale. Thus, from the Doob inequality, the Cauchy--Schwarz inequality,
Lemmas \ref{moments}--\ref{lem5.2} and the above bound on $h$ we obtain
\begin{eqnarray*}
&E_P[J^2_2]= O(1)\sum_{i=0}^{n-1} E_P\left(h^2(i T/n, B^2_{\theta_i})\left(\theta_{i+1}-\theta_i-E_P(\theta_{i+1}-\theta_i|\mathcal F_{\theta_i})\right)^2 \right) \\
&=O(1)\sum_{i=0}^{n-1}\left(E_P\left(h^4(i T/n, B^2_{\theta_i})\right)\right)^{1/2}
 \left(E_P\left(\left(\theta_{i+1}-\theta_i-E_P(\theta_{i+1}-\theta_i|\mathcal F_{\theta_i})\right)^4\right)\right)^{1/2}\\
 &=O((1+s)^2 n^{-1}).
\end{eqnarray*}
From the Jensen inequality we conclude that $E_P[J_2]=O((1+s) n^{-1/2})$.

Finally, we estimate $E_P[\theta_n J_3].$
From (\ref{2.function})
and the fact that
$\frac{B_t}{B_{\theta_k}}=1+O(1/\sqrt n)$
for $t\in [\theta_k,\theta_{k+1}]$
it follows that
\[J_3\leq
O(n^{-1/2})\max_{0\leq k\leq n-1}B_{\theta_k}+O(1)\max_{0\leq k\leq n-1}[(1+B_{\theta_k})\sup_{\theta_k\leq t\leq\theta_{k+1}}|t-kT/n|].\]
Observe that $\max_{0\leq k\leq n-1}\sup_{\theta_k\leq t\leq\theta_{k+1}}|t-kT/n|\leq T/n+\max_{1\leq k\leq n}|\theta_k-k T/n|$.
This together with the Cauchy--Schwarz inequality, Lemma \ref{moments} and Lemma \ref{lem5.3} gives
\begin{eqnarray*}
&E_P[\theta_n J_3]=O(n^{-1/2})\left(E_P\left[\theta^2_n\right]\right)^{1/2} \left(E_P\left(\max_{0\leq k\leq n-1}B^2_{\theta_k}\right)\right)^{1/2}+\\
&O(1)\frac{T}{n}\left(E_P\left[\theta^2_n\right]\right)^{1/2}\left(E_P\left(\max_{0\leq k\leq n-1}(1+B_{\theta_k})^2\right)\right)^{1/2}+\\
&O(1)\left(E_P\left[\theta^2_n\right]\right)^{1/2}\left(E_P\left(\max_{0\leq k\leq n-1}(1+B_{\theta_k})^4\right)\right)^{1/4}
\left(E_P\left(\max_{1\leq k\leq n}|\theta_k-k T/n|^4\right)\right)^{1/4}\\
&=O\left((1+s)n^{-1/2}\right)
\end{eqnarray*}
and the proof is completed.
\end{proof}
\section{Game Options and Numerical Results}\label{sec:6}\setcounter{equation}{0}
In this section we apply Theorem \ref{thm2.1} and provide
numerical analysis for
path--independent game options with the payoffs
$Y_t=f(t,B_t)$ and $X_t=f(t,B_t)$, $t\in [0,T]$, and we set $Z\equiv 0$.
First (for the above payoffs), we establish the connection between the super--hedging price of game options and Dynkin games,
in the model uncertainty setup.
\subsection{Game Options}
A game contingent claim (GCC) or game option, which was introduced in
\cite{K}, is defined
as a contract between the seller and the buyer of the option such
that both have the right to exercise it at any time up to a
maturity date (horizon) $T$. We consider the following GCC with Markovian payoffs.
If the buyer exercises the contract
at time $t$ then he receives the payment $Y_t=f(t,B_t)$, but if the seller
exercises (cancels) the contract before the buyer then the latter
receives $X_t=g(t,B_t)$. The difference $X_t-Y_t$ is the penalty
which the seller pays to the buyer for the contract cancellation.
In short, if the seller will exercise at a stopping time
$\gamma\leq{T}$ and the buyer at a stopping time $\tau\leq{T}$
then the former pays to the latter the amount
$H(\gamma,\tau)$ given by (\ref{1.1}).

Next, we introduce the setup of super--hedging for the seller (the buyer setup is symmetrical).
Recall the natural filtration, $\mathcal F=\mathcal F_t$, $t\geq 0$. We denote by
 $L(B,\mathcal P^{(I)}_s)$ the set of all $\mathcal F$--predictable processes $\Delta=\{\Delta_t\}_{t=0}^T$
 such that for any $P\in\mathcal P^{(I)}_s$, the stochastic (It\^{o}) integral $\int_{0}^t \Delta_u dB_u$, $t\in [0,T]$ is well defined and a super--martingale
with respect to $\mathcal F$.
We define a hedge for the seller as a triplet $(x,\Delta,\gamma)\in \mathbb R\times L(B,\mathcal P^{(I)}_s)\times\mathcal T_T$ which consists of
an initial capital $x$, a trading strategy
$\Delta=\{\Delta_t\}_{t=0}^T$ and a stopping time $\gamma$. A hedge
$(x,\Delta,\gamma)$ is perfect if for any stopping time (for the buyer) $\tau\in\mathcal T_T$ we have the inequality
$$x+\int_{0}^{\gamma\wedge\tau} \Delta_u dB_u\geq H(\gamma,\tau) \ \ P-\mbox{a.s.} \ \mbox{for} \ \mbox{all} \ \ P\in\mathcal P^{(I)}_s.$$
The super--hedging price is defined by
$$\mathbf V:=\inf\{x\in\mathbb R: \ \exists (\Delta,\gamma) \ \mbox{such} \ \mbox{that} \ (x,\Delta,\gamma) \ \mbox{is}
\ \mbox{a} \ \mbox{perfect} \ \mbox{hedge}\}. $$
\begin{lem}\label{lem6.1}
The super--hedging price is given by
$\mathbf V=V^{(I)}(0,s).$ Moreover, there exists a perfect hedge with initial capital $V^{(I)}(0,s)$.
\end{lem}
\begin{proof}
As usual, the inequality $\mathbf V\geq V^{(I)}(0,s)$ is immediate. Indeed if
$(x,\Delta,\gamma)$ is a perfect hedge then from the super--martingale property of
$\int_{0}^t \Delta_u dB_u$, $t\in [0,T]$ we obtain that
for any $\tau\in \mathcal T_T$ and $P\in\mathcal P^{(I)}_s$
$$x\geq E_P\left[x+\int_{0}^{\gamma\wedge\tau} \Delta_u dB_u\right]\geq  E_P [H(\gamma,\tau)].$$
Thus $x\geq V^{(I)}(0,s)$
as required.

It remains to show that there exists a perfect hedge with initial capital
$V^{(I)}(0,s)$.
We apply Theorem 4.1 in \cite{BY} which not only gives the optimal stopping time for the
player which plays against nature but also
a sub--martingale property up to the optimal time. Once again taking Remark \ref{rem2.1} into account, for our setup
the sub--martingale property
becomes
a super--martingale property.
More precisely, Theorem 4.1 in \cite{BY} implies that for
the stopping time $\gamma^{*}:=T\wedge\inf\{t: X_t=V^{(I)}(t,B_t)\}$ we have the following property.
For any $P\in\mathcal P^{(I)}_s$,
the process
$V^{(I)}(t\wedge\gamma^{*},B_{t\wedge\gamma^{*}})$, $t\in [0,T]$ is a $P$--super--martingale
with respect to
the natural filtration $\mathcal F_t$, $t\geq 0$.

We apply the nondominated version of the optional decomposition theorem .
Since quadratic variation can be defined in a pathwise form then the condition
$B^{-1}\sqrt\frac{d\langle  B\rangle }{dt}\in I$ is invariant under equivalent change of measure. Hence the
set
$P^{(I)}_s$ is a saturated set (using \cite{Nu} terminology) of martingale measures. Namely, if
$P\in \mathcal P^{(I)}_s$ and $Q\sim P$ is a martingale measure on the canonical space then
$Q\in \mathcal P^{(I)}_s$. Thus, from Theorem 2.4 in \cite{Nu} it follows that there exists a process
$\Delta^{*}\in L(B,\mathcal P^{(I)}_s)$ such that for any probability measure $P\in\mathcal P^{(I)}_s$
\begin{equation}\label{6.1-}
P\left(V^{(I)}(0,s)+\int_{0}^{t}\Delta^{*}_u dB_u- V^{(I)}(t\wedge\gamma^{*},B_{t\wedge\gamma^{*}})\geq 0, \ \ \forall t\in [0,T]\right)=1.
\end{equation}
We claim that $(V^{(I)}(0,s),\Delta^{*},\gamma^{*})$ is a perfect hedge. Indeed,
let $\tau\in\mathcal T_T$ be a stopping time for the buyer and $P\in\mathcal P^{(I)}_s$.
First consider the event $\{\gamma^{*}<\tau\}$.
On this event we have (recall the definition of $\gamma^{*}$)
$V^{(I)}(\tau\wedge\gamma^{*},B_{\tau\wedge\gamma^{*}})=X_{\gamma^{*}}=H(\gamma^{*},\tau)$ and so from (\ref{6.1-})
\begin{equation*}
V^{(I)}(0,s)+\int_{0}^{\tau\wedge\gamma^{*}}\Delta^{*}_u dB_u\geq H(\gamma^{*},\tau) \ \ P-\mbox{a.s.}
\end{equation*}
Finally, we consider the event $\{\gamma^{*}\geq\tau\}$. Applying (\ref{6.1-}) and the trivial inequality
$V^{(I)}(t,x)\geq f(t,x)$ for all $t,x$
we obtain
\begin{equation*}
V^{(I)}(0,s)+\int_{0}^{\tau\wedge\gamma^{*}}\Delta^{*}_u dB_u\geq Y_{\tau}=H(\gamma^{*},\tau) \ \ P-\mbox{a.s.}
\end{equation*}
and the proof is completed.
\end{proof}
\begin{rem}
It seems that by applying Theorem 4.1 in \cite{BY} and the optional decomposition Theorem 2.4 in \cite{Nu}, Lemma 6.1 can be extended to path dependent options as long
as the regularity assumptions from \cite{BY} are satisfied. Since we are motivated by numerical applications, then for simplicity
we considered path--independent payoffs.
Still, a challenging open question, is whether Lemma \ref{lem6.1} can be
obtained under weaker (than Lipschitz or uniform type of continuity) regularity conditions.
\end{rem}

\subsection{Numerical Results}
In view of Lemma \ref{lem6.1} we use Theorem \ref{thm2.1} and provide a numerical analysis for
the super--hedging price of path--independent game options.
We assume that the interest rate in the market is a constant $r>0$, and so the stock price before discounting is
given by $S_t=e^{rt} B_t$, where, recall that $B$ is
the canonical process.
The payoffs before discounting are of the form
$\hat X_t=\hat g(S_t)$, $\hat Y_t=\hat f(S_t)$
where $\hat g\geq \hat f$. In order to compute the game option price we need to consider the discounted payoffs and so
during this section we put
$g(t,x):=e^{-rt}\hat g(e^{rt} x)$, $f(t,x):=e^{-rt}\hat f(e^{rt} x)$ and
$h\equiv 0$.

In \cite{E} (see Section 4),
the author proved that for game options (with finite or infinite maturity)
with continuous path--independent payoffs $\hat g,\hat f$
satisfying
\begin{equation}\label{6.1}
\frac{\hat g(x)}{x},\frac{\hat f(x)}{x} \  \mbox{are} \ \mbox{non} \  \mbox{increasing} \ \mbox{for} \ x>0
\end{equation}
the price is non decreasing in the volatility. Thus, (if the above assumption is satisfied) the price under volatility uncertainty
which is given by the interval $I=[\underline{\sigma},\overline{\sigma}]$ is the same as the price in the complete Black--Scholes market
with a constant volatility $\overline{\sigma}$. The later value can be approximated by the standard binomial models (see \cite{Ki1}).
In particular, this is the case for game put options given by
$$\hat g(x)=C(K-x)^{+}+\delta  \ \ \mbox{and} \ \  \hat f(x)=(K-x)^{+}, \ \ C\geq 1,\ K,\delta>0.$$
In Table 1, we test numerically the above statement from \cite{E} for game put options. This is done by comparing our numerical results with previous numerics
which was obtained in \cite{KKS} for
game put options in the Black--Scholes model.
\begin{table}
\begin{center}
\caption{
In this table we take the parameters
$r = 0.06$, $T = 0.5$, $K=100$, $\delta = 5$ and provide numerical results for game put options under model uncertainty given by the
interval $I=[0,0.4]$. We compare our results to previous numerical results (see \cite{KKS}) for game put options in the
Black--Scholes model with volatility $\sigma=0.4$.}
\begin{tabular}{lccccc}
\hline
\multicolumn{5}{c}{Values obtained with }\\
\cline{2-5}
$S_0$ & n = 200 & n = 400 & n = 700 & n = 1200 & Black--Scholes with $\sigma=\overline{\sigma}$ \\
\hline
80 & 20.7003 & 20.6719 & 20.6593 &  20.6532 & 20.6   \\
90 & 12.4932 & 12.4787 & 12.4938 & 12.4683 &  12.4 \\
100 & 5.00 & 5.00 & 5.00 & 5.00 & 5.00 \\
110 & 3.7609 & 3.7240 & 3.6862 & 3.6916 &  3.64 \\
120 & 2.6169 & 2.5897 & 2.5822 & 2.5729 &  2.54  \\
\hline
\end{tabular}
\end{center}
\end{table}

\subsubsection*{Game call options}
Next, we deal with game call options given by
$$\hat g(x)=C(x-K)^{+}+\delta  \ \ \mbox{and} \ \  \hat f(x)=(x-K)^{+}, \ \ C\geq 1,\ K,\delta>0.$$
We observe that in this case (\ref{6.1}) is not satisfied and so we expect that the price for
the model uncertainty interval $I=[\underline{\sigma},\overline{\sigma}]$ will be strictly bigger than the game call option price in the Black--Scholes model with volatility
$\overline{\sigma}$.
We take $C=1$, namely we consider game call options with constant penalty.

First, we compare (Table 2) the option prices
under model uncertainty with the prices in the Black--Scholes model (with the highest volatility).
Since we could not find previous numerical results for finite maturity game call options
in the Black--Scholes model, we compute it by applying the binomial trees from
\cite{Ki1}.
These trees are "almost" the same as our trees for the case where the volatility uncertainty interval $I$ contains only one point.
We observe that for call options
the prices in general should not coincide.
\begin{table}
\begin{center}
\caption{
We take the same parameters as in Table 1
and provide numerical results for game call options under model uncertainty given by the
interval $I=[0,0.4]$. We compare our results to binomial approximations for the Black--Scholes model with $\sigma=0.4$.}
\begin{tabular}{lcccc}
\hline
\multicolumn{5}{c}{Values obtained under model uncertainty }\\
\cline{2-5}
$S_0$ & n = 200 & n = 400 & n = 700 & n = 1200  \\
\hline
80 & 2.0805  & 2.0893 & 2.0847 & 2.0948 \\
85 & 2.8138 & 2.7964 & 2.8055 &  2.8018    \\
90 & 3.6553 & 3.5966 & 3.6241 & 3.6064 \\
95 & 4.5827 & 4.4682 & 4.5050 & 4.4874  \\
105 & 5.00 & 5.00 & 5.00 & 5.00  \\
110 & 10.00 & 10.00 & 10.00 & 10.00 \\
115 & 15.00 & 15.00 & 15.00 & 15.00   \\
120 & 20.00 & 20.00 & 20.00 & 20.00   \\
\hline
\multicolumn{5}{c}{Values obtained for Black--Scholes  }\\
\cline{2-5}
$S_0$ & n = 200 & n = 400 & n = 700 & n = 1200  \\
\hline
80 & 2.0625 & 2.0359 & 2.0244& 2.0210 \\
85 & 2.7706 & 2.7301 & 2.7274 &  2.7143    \\
90 & 3.5066 & 3.4889 & 3.4968 & 3.4798 \\
95 & 4.3497 & 4.3124 & 4.3056 & 4.2481  \\
105 & 5.00 & 5.00 & 5.00 & 5.00  \\
110 & 10.00 & 10.00 & 10.00 & 10.00 \\
115 & 14.9355 & 14.9304 & 14.9275 & 14.9260   \\
120 & 19.7812 & 19.7735 & 19.7691 & 19.7669   \\
\end{tabular}
\end{center}
\end{table}

Finally, we calculate numerically the stopping regions.
We observe that the discounted payoff
$f(t,B_t)=(B_t-Ke^{-rt})^{+}$, $t\geq 0$ is a
sub--martingale with respect to any probability measure
in the set $\mathcal P^{(I)}_s$. Thus, the buyer's optimal stopping time is just
$\tau\equiv T$.

For the seller, the optimal stopping time is (see Theorem 4.1 in \cite{BY})
$$\gamma^{*}=T\wedge\inf\{t: g(t,B_t)=V^{(I)}(t,B_t)\}.$$
Introduce the function
$$\tilde V(u,x):=\sup_{P\in \mathcal{P}^{(I)}_x}\sup_{\tau\in\mathcal T_{u}}\inf_{\gamma\in\mathcal T_{u}}
E_{P}\left[e^{-r(\tau\wedge\gamma)}\left((S_{\tau\wedge\gamma}-K)^{+}+\delta\mathbb{I}_{\gamma<\tau}\right)\right]
$$
where as before $S_t=e^{rt} B_t$, $t\geq 0$ is the stock price. The term $\tilde V(u,x)$
is the price of a game call option with maturity date $u$ and initial stock price $S_0=x$.
We observe that
$\gamma^{*}=T\wedge\inf\{t: S_t\in D\}$, where $D=D(T)$ is the stopping region (of course it depends on the maturity date $T$) given by
$$D=\{(t,x): \tilde V(T-t,x)=(x-K)^{+}+\delta\}.$$

In \cite{YYZ}, the authors studied the structure of the stopping region $D$ for game call options in the complete Black--Scholes market.
They proved (see Theorem 4.2) that the stopping region $D$ is of the form
$$D=\{(t,x): t\in [0,T_1], \ K\leq x\leq b(t)\}\bigcup \left\{[T_1,T_2]\times\{K\}\right\}$$
where $T_1<T_2<T$ and $b:[0,T_1]\rightarrow [K,\infty)$ can be computed numerically.

In Figure 1 we calculate numerically the stopping regions (for the seller)
for game call options both in the model uncertainty setup given by the interval $I=[0,0.4]$
and in the complete Black--Scholes model with volatility $\sigma=0.4$.
We obtain that the structure from \cite{YYZ} is valid for the model uncertainty case as well.
Furthermore, for both cases $T_2$ is the same, while $T_1$ and $b$ are different.
Up to date, there is no theoretical results related to the explicit structure of
stopping regions for game options under model uncertainty.

\begin{figure}
\centering
\includegraphics[width=0.9\textwidth]{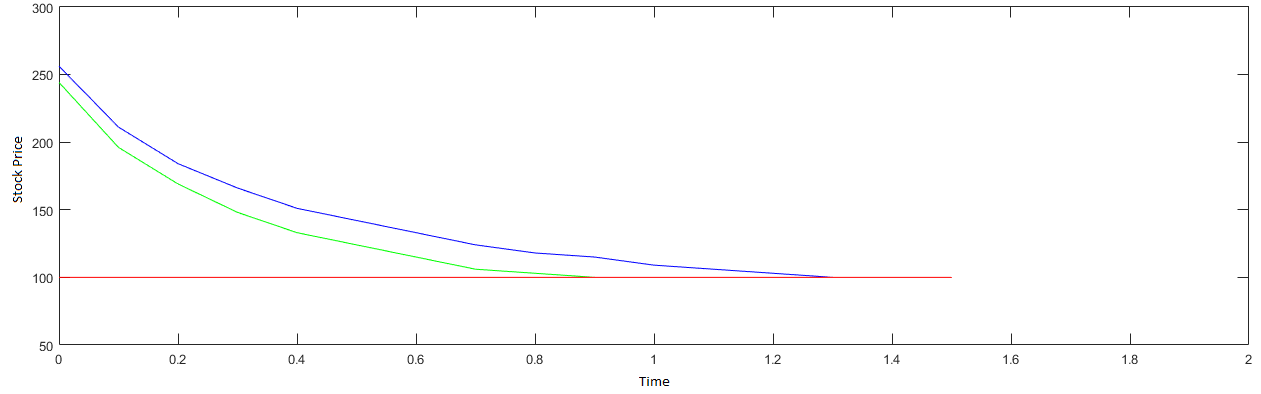}
\caption{We consider a game call option with maturity date $T=2$, a constant penalty $\delta=12$ and a strike price $K=100$.
As before the interest rate is $r=0.06$.
We take $n=1200$ and compute numerically the stopping regions for the seller. For the model uncertainty given by the interval $I=[0,0.4]$
we get that for $t\in [0,1.3]$ the seller should exercise at the first moment when the stock price is between the strike price and the
value given by the blue curve. For $t\in [1.3,1.5]$ the seller stops at the first moment the stock price equals to the strike price.
After the time $t=1.5$ the investor should not exercise (before the maturity date).
For the Black--Scholes model with volatility $\sigma=0.4$ we get that
for $t\in [0,0.9]$ the seller should exercise at the
first moment when the stock price is between the strike price and the
value given by the green curve. For $t\in [0.9,1.5]$ the seller stops at the first moment the stock price equals to the strike price.
After the time $t=1.5$ the investor should not exercise (before the maturity date).}
\end{figure}

\end{document}